\documentclass[11pt,twoside,reqno]{amsart}
\allowdisplaybreaks
\usepackage{amsmath,amstext,amssymb,epsfig}
\usepackage{graphicx}
\usepackage{mathrsfs}
\calclayout
\makeatletter
\renewcommand{\@seccntformat}[1]{\bf\csname the#1\endcsname.}
\renewcommand{\section}{\@startsection{section}{1}
	\z@{.7\linespacing\@plus\linespacing}{.5\linespacing}
	{\normalfont\upshape\bfseries\centering}}
\renewcommand{\@biblabel}[1]{\@ifnotempty{#1}{#1.}}
\makeatother

\theoremstyle{plain}
\newtheorem{thm}{Theorem}[section]
\newtheorem{lem}[thm]{Lemma}
\newtheorem{proposition}[thm]{Proposition}
\newtheorem{cor}[thm]{Corollary}

\theoremstyle{definition}

\newtheorem{definition}[thm]{Definition}
\newtheorem{rem}{Remark}[section]

\usepackage{cancel}
\usepackage[parfill]{parskip}
\usepackage[german]{varioref}
\usepackage[all]{xy}
\usepackage{color}

\def \>{\succ}
\def \<{\prec}

\begin{document}	

\title[Imed Basdouri \textsuperscript{1},Mariem Jendoubi\textsuperscript{2}, Ahmed Zahari Abdou Damdji.\textsuperscript{3}]{Related Hom-rhizaform algebras and Rota-Baxter Operators, Hom-Rhizaform  Family Algebras}
	\author{Imed Basdouri\textsuperscript{1},Mariem Jendoubi\textsuperscript{2},  Ahmed Zahari Abdou Damdji.\textsuperscript{3}}

		\address{\textsuperscript{1}D\'epartement de Math\'ematiques, Faculty of Sciences,University of Sfax,Sfax,Tunisia}
        \address{\textsuperscript{2}D\'epartement de Math\'ematiques, Faculty of Sciences,University of Sfax,Sfax,Tunisi}
        \address{\textsuperscript{3}IRIMAS-Department of Mathematics, Faculty of Sciences, University of Haute Alsace, Mulhouse, France}
	
\email{\textsuperscript{1}mariemj090@gmail.com}
\email{\textsuperscript{2}basdourimed@yahoo.fr}
\email{\textsuperscript{3}abdou-damdji.ahmed-zahari@uha.fr}
	
	
	\date{\today}
\begin{abstract} 
This paper explores the link between Hom-rhizaform algebras and Rota-Baxter operators. We define a new structure, the Hom-rhizaform family algebra, which is a more general version of the 
Hom-rhizaform algebra. The main finding is that Rota-Baxter operators can be used to construct new Hom-rhizaform algebras. This work expands the theory of Hom-algebras by showing a new 
way to apply Rota-Baxter operators. Finally, we establish the classification of these algebras as well as their corresponding cocycle cones.
 \end{abstract}
\maketitle 

\section{Introduction}\label{introduction}
The concept of Rota-Baxter operators on associative algebras was introduced in the 1960th by G.Baxter \cite{GB} in the study of fluctuation theory in Probability. Uchino introduced the notion of generalized Rota-Baxter operator (called also O-operator) on a bimodule over an associative algebra in \cite{A,Makhlouf2020,AO,Zahari2025, Basdouri2025}. 
Rota-Baxter operators and their generalization O-operators are
 becoming rather popular due to the irinterest in various domain like nonassociative algebra and applied algebra.
Let $(M,r,l,\beta)$ be an $\mathscr{A}$-bimodule, and let $\pi : M \to \mathscr{A}$ a linear map is defined to be the identity, 

\begin{equation}\label{eq001}
\pi(m) \cdot \pi(n) = \pi\big( l(\pi(m))\,n + r(\pi(n))\,m \big).
\end{equation}

 When an operator $\pi : M \to A $ satisfies the condition \eqref{eq001}, we say that $\pi$ is a generalized Rota-Baxter operator, or shortly, $O$-operator. When $M = \mathscr{A}$  is reduced 
to a Rota-Baxter condition  $R(x) \cdot R(y) = R(R(x) \cdot y +x\cdot R(y))$. The $R : \mathscr{A} \to \mathscr{A}$ is called a Rota-Baxter operator of “weight zero”,
or simply Rota-Baxter operator. In general, Rota-Baxter operators are defined by the weighted formula, $R(x)\cdot R(y)=R(R(x)\cdot y+x\cdot R(y))-q\cdot R(x\cdot y)$,
where $q$ is a scalar (weight). The Rota-Baxter operators of $q=0$ are special examples. In the following, we assume the weight of Rota-Baxter operator is zero.
 Rhizaform algebras are algebras with two operations, whose sum satisfies an anti-associative identity (i.e. $(x\cdot y)\cdot z=-x\cdot(y\cdot z)$). Moreover, 
the left and right multiplication operators forming bimodules over the anti-associative algebra $(\mathscr{A},\cdot)$. This characterization parallels that of dendriform algebras,
 where the sum of operations preserves associativity. In \cite{R, IA, Makhlouf2020}, it was explained how one may associate a rhizaform algebra to any antiassociative algebra equipped with a 
Rota-Baxter operator, i.e., it was shown that if $\beta :\mathscr{A}\to \mathscr{A}$ is a Rota-Baxter operator then the pair of multiplications $x\succ y:=\beta(x)y$ and 
$x\prec y:=x\beta(y)$ is a rhizaform structure on A. The operations on an algebra may be replaced by a family of operations indexed  by a semigroup, giving rise to the same algebraic 
structure in family versions. The notion of Rota-Baxter family appeared in 2007 is the first example of this  situation \cite{I, B, GB} (see also \cite{LG, Zahari2023, BR}) about algebraic aspects of renormalization  in Quantum  Field Theory. This terminology was suggested to the authors by Li Guo, who  further duscussed the underlying structure under the name Rota-Baxter Family
 algebra \cite{I, Mosbahi2023, BW} which is a generalization of Rota-Baxter algebras. Since then, various other kinds of family algebraic structures have been defined
 \cite{F, YZ1, YZ2, YD, Zahari2023}.
This article aims to extend the notion of rhizaform algebras and rhizaform family algebras to the Hom-setting.
Recall that the theory of Hom-algebras first appeared in the study of quasi-deformations of Lie algebras of vector fields, including q-deformations of Witt algebras and Virasoro algebras with the introduction of Hom-Lie algebras by J.T. Hartwig, D. Larsson, and S.D. Silvestrov \cite{HL, HJ} which offer an approach to addressing anti associativity. 
The paper is organized as follows : In section two, we introduce the notions of hom-rhizaform algebras as an approach of splitting the Hom-anti-associativity.
The notions of O-operators and Rota-Baxter operators on Hom-anti-associative algebras are introduced to interpret hom-rhizaform algebras in section 3. In section 4 The relationships between hom-rhizaform algebras, commutative Connes cocycles on Hom-anti-associative algebras are given. In the Section 5, we introduce the notions of nilpotency for Hom-rhizaform algebras. Moreover, we give the classification of 2-dimensional and the classification of 3-dimensional Hom-rhizaform algebras over the field of complex numbers. In the last Section 5, we introduce the concepts of Hom-rhizaform family algebras and  Hom-anti-associative family algebras, and we provide the relations between these family Hom-algebras.
\section{Hom-anti-associative and Hom-rhizaform algebras}
\begin{definition}
A Hom-algebra $(\mathscr{A}, \ast , \alpha)$ over a field  $\mathbb{F}$ is called a Hom-Jacobi-Jordan algebra if it  satisfies  the following two identities 

\begin{align}
x\ast y-y\ast x=0\label{eq1},\\
\alpha(x) \ast (y \ast z)+ \alpha(y)\ast(z\ast x)+\alpha(z)\ast(x\ast y)=0\label{eq2},
\end{align}
for all $x, y,z\in\mathscr{A}.$
\end{definition}
\begin{definition}
  $(\mathscr{A}, \circ, \alpha)$ is called a \textit{Hom-Jacobi-Jordan-admissible algebra}, where $\mathscr{A}$ is a vector space equipped with a bilinear operation $\circ : \ \mathscr{A} \times \mathscr{A} \to \mathscr{A}$ and a linear map $\alpha :\ \mathscr{A}\to \mathscr{A}$, if the bilinear operation $[,] : \ \mathscr{A} \times \mathscr{A} \to \mathscr{A}$ defined by
\[
[x, y] = x \circ y + y \circ x, \quad \forall x, y \in \mathscr{A},
\]
makes $(\mathscr{A}, [,], \alpha)$ a Hom-Jacobi-Jordan algebra. This algebra $(\mathscr{A}, [,], \alpha)$ is called the \textit{Hom sub-adjacent Jacobi-Jordan algebra} of $(\mathscr{A}, \circ, \alpha)$ and is denoted by $(\mathcal{J}(A), [,], \alpha)$.
\end{definition}

\begin{definition}
A Hom-pre-Jacobi-Jordan algebra $(\mathscr{A},\ \cdot,\ \alpha)$ is a vector space with a binary operation\\ $(x,\ y)\mapsto xy$  satisfying 

\begin{align}
(xy)\alpha(z)+\alpha(x)(yz)+(yx)\alpha(z)+\alpha(y)(xz)=0\label{eq3}
\end{align}
for all $x, y,z\in\mathscr{A}.$
\end{definition}
\begin{definition}
  An Hom anti-associative algebra $(\mathscr{A},\ \cdot,\ \alpha)$ is a nonassociative algebra whose multiplication satisfies the identity   
  \begin{align*}
 \alpha(x) \cdot (y \cdot z) = - (x \cdot y) \cdot \alpha(z), \quad \forall x,\ y,\ z \in\mathscr{A}.
  \end{align*}
\end{definition}
\begin{definition}
 Let \( \mathscr{A} \) be a vector space with two bilinear operations 
\[
\succ: \ \mathscr{A} \times \mathscr{A} \to \mathscr{A}, \quad \prec:\ \mathscr{A} \times \mathscr{A} \to \mathscr{A},\] and a linear map $\alpha:\ \mathscr{A}\to \mathscr{A}$. Define a bilinear operation \( \ast \) as
\begin{equation}
    x \ast y = x \succ y + x \prec y, \quad \forall x, y \in \mathscr{A}.\label{eq0}
\end{equation}
The triple \( (\mathscr{A}, \succ, \prec,\alpha) \) is called an \emph{Hom-anti-associative admissible algebra} if \( (\mathscr{A}, \ast,\alpha) \) is an Hom-anti-associative algebra. In this case, \( (\mathscr{A}, \ast,\alpha) \) is called the \emph{associated anti-associative algebra} of \( (\mathscr{A}, \succ, \prec,\alpha) \).
\end{definition}
\begin{rem}
The quadriple $(\mathscr{A}, \prec, \succ, \alpha)$ is a Hom-anti-associative admissile algebra if and only if the following equation holds : 
\begin{align}
&(x\succ y)\succ\alpha(z)+(x\prec y)\succ\alpha(z)+(x\succ y)\prec(\alpha(z)+(x\prec y)\prec \alpha(z)\label{eq4}\\
&=-\alpha(x)\succ(y \succ z)-\alpha(x)\succ(y\prec z)-\alpha(x)\prec(y\succ z)- \alpha(x)\prec(y\prec z),\,\forall\, x,y,z\in \mathscr{A}.\nonumber 
\end{align}
\end{rem}

\begin{definition}\label{def0}
A  Hom-rhizaform algebra is a quadruple $(\mathscr{A}, \prec, \succ, \alpha)$ consisting of a vector space $\mathscr{A}$,two linear operations$
\succ: \ \mathscr{A} \otimes \mathscr{A} \to \mathscr{A}, \quad \prec:\  \mathscr{A} \otimes A \to A$ and a linear map $\alpha:\ \mathscr{A}\to \mathscr{A}$  satisfying

\begin{align}
\alpha\circ\prec=\prec\circ\alpha^{\otimes}\,&\text{and}\,\alpha\circ\succ=\succ\circ\alpha^{\otimes} \\
(x\succ y+x\prec y)\succ\alpha(z)&=-\alpha(x)\succ(y\succ z) \label{req1}\\
\alpha(x) \prec (y \succ z + y \prec z)&=-(x\prec y)\prec\alpha(z)\label{req2},\\
 \alpha(x) \succ (y \prec z)&=-(x\succ y) \prec \alpha(z)\label{eq8},
\end{align}
for all $x,y,z\in \mathscr{A}.$
We called $\alpha$ ( in this order ) the structure maps of $\mathscr{A}$.
\end{definition}

\begin{definition}
A Hom-dendriform algebra is a quadruple $(\mathscr{A}, \prec, \succ, \alpha)$ consisting of a vector space $\mathscr{A}$,two linear operations 
$\succ: \ \mathscr{A} \otimes \mathscr{A} \to A, \quad \prec:\ \mathscr{A} \otimes \mathscr{A} \to \mathscr{A}$ and a linear map $\alpha:\ \mathscr{A}\to \mathscr{A}$ satisfying
\begin{align}
\alpha\circ\prec=\prec\circ\alpha^{\times}\,&\text{and}\,\alpha\circ\succ=\succ\circ\alpha^{\times} \\
(x \prec y) \prec \alpha(z)&=\alpha(x) \prec (y \prec z + y \succ z)\label{eq5},\\
 (x \succ y) \prec \alpha(z)&=\alpha(x) \succ (y \prec z)\label{eq6},\\
 (x \prec y + x \succ y) \succ \alpha(z)&=\alpha(x) \succ (y \succ z)\label{eq7}.
\end{align}
for all $x,y,z\in \mathscr{A}.$
We called $\alpha$ ( in this order ) the structure maps of $\mathscr{A}$.
\end{definition}

\begin{definition}
If $(\mathscr{A}_1, \succ_1, \prec_1, \alpha_1)$, $(\mathscr{A}_2, \succ_2, \prec_2, \alpha_2)$  are Hom-rhizaform (Hom-dendriform) algebra. 
Then, a fonction $\pi$ from $\mathscr{A}_1\longrightarrow \mathscr{A}_2$ is homomorphism if $\alpha_2\circ \pi=\pi\circ \alpha_1$,
$$
\pi(x\succ_1 y)=\pi(x)\succ_2\pi(y)\quad \txt{and} \quad \pi(x\prec_1 y)=\pi(x)\prec_2\pi(y)\, \forall\, x,y\in\mathscr{A}. 
$$
\end{definition}

\begin{thm}
  Let $(\mathscr{A}, \succ, \prec, \alpha)$ be a Hom-rhizaform algebra. Then the following statements hold:
\begin{itemize}
    \item[(i)] Define a bilinear operation $\ast$ by Equation \eqref{eq0}. Then $(\mathscr{A}, \ast, \alpha)$ is a Hom-anti-associative algebra, called the \textit{associated Hom-anti-associative algebra} of $(\mathscr{A}, \succ, \prec, \alpha)$. Furthermore, $(\mathscr{A}, \succ, \prec, \alpha)$ is called a \textit{compatible Hom-rhizaform algebra structure} on $(\mathscr{A}, \ast, \alpha)$.
    
    \item[(ii)] The bilinear operation $\circ : \mathscr{A} \times \mathscr{A} \to \mathscr{A}$ given by
    \[
    x \circ y = x \succ y - y \prec x, \quad \forall x, y \in \mathscr{A},
    \]
    defines a Hom-pre-Jacobi-Jordan algebra, called the \textit{Hom-associated pre-Jacobi-Jordan algebra} of $(\mathscr{A}, \succ, \prec, \alpha)$.
    
    \item[(iii)] Both $(\mathscr{A}, \ast, \alpha)$ and $(\mathscr{A}, \circ, \alpha)$ have the same Hom-sub-adjacent Jacobi-Jordan algebra $(\mathfrak{g}(\mathscr{A}), [,], \alpha)$ defined by
    \[
    [x, y] = x \succ y + x \prec y + y \succ x + y \prec x, \quad \forall x, y \in \mathscr{A}.
    \]
\end{itemize}  
\end{thm}
\begin{proof}
  \begin{itemize}
    \item[(i)] This is straightforward.
 \item[(ii)] Let $x, y, z \in \mathscr{A}$. Then we have
    \begin{align*}
    (x \circ y) \circ \alpha(z) 
    &= (x \succ y - y \prec x) \succ \alpha(z) - \alpha(z) \prec (x \succ y - y \prec x) \\
    &= (x \succ y) \succ \alpha(z) - (y \prec x) \succ \alpha(z) - \alpha(z) \prec (x \succ y) + \alpha(z) \prec (y \prec x), \\
    \alpha(x) \circ (y \circ z) 
    &= \alpha(x) \succ (y \succ z - z \prec y) - (y \succ z - z \prec y) \prec \alpha(x) \\
    &= \alpha(x) \succ (y \succ z) - \alpha(x) \succ (z \prec y) - (y \succ z) \prec \alpha(x) + (z \prec y) \prec \alpha(x).
    \end{align*}

    By swapping $x$ and $y$, we have
    \begin{align*}
    (y \circ x) \circ \alpha(z) 
    &= (y \succ x) \succ \alpha(z) - (x \prec y) \succ \alpha(z) - \alpha(z) \prec (y \succ x) + \alpha(z) \prec (x \prec y), \\
    \alpha(y) \circ (x \circ z) 
    &= \alpha(y) \succ (x \succ z) - \alpha(y) \succ (z \prec x) - (x \succ z) \prec \alpha(y) + (z \prec x) \prec \alpha(y).
    \end{align*}

    Using Eqs. \ref{req1},\ref{req2},\ref{eq8} we obtain
    \begin{align*}
    &\alpha(x) \circ (y \circ z) + \alpha(y) \circ (x \circ z) \\
    &= \alpha(x) \succ (y \succ z) - \alpha(x) \succ (z \prec y) - (y \succ z) \prec \alpha(x) + (z \prec y) \prec \alpha(x) \\
    &\quad + \alpha(y) \succ (x \succ z) - \alpha(y) \succ (z \prec x) - (x \succ z) \prec \alpha(y) + (z \prec x) \prec \alpha(y) \\
    &= \alpha(x) \succ (y \succ z) + (z \prec y) \prec \alpha(x) + \alpha(y) \succ (x \succ z) + (z \prec x) \prec \alpha(y) \\
    &= - (x \succ y + x \prec y) \succ \alpha(z) - \alpha(z) \prec (y \succ x + y \prec x) \\
    &\quad - (y \succ x + y \prec x) \succ \alpha(z) - \alpha(z) \prec (x \succ y + x \prec y) \\
    &= - (y \circ x) \circ \alpha(z) - (x \circ y) \circ \alpha(z).
    \end{align*}
    Moreover, we have
    \[
    x \circ y + y \circ x = x \succ y + y \prec x + x \prec y + y \succ x = x \ast y + y \ast x, \quad \forall x, y \in \mathscr{A}.
    \]

    Thus, $(\mathscr{A}, \circ, \alpha)$ is a Hom-Jacobi-Jordan-admissible algebra and hence a Hom-pre-Jacobi-Jordan algebra.
    \item[(iii)] This is straightforward. Note that it also appears in the proof of (ii).
\end{itemize}  
	This ends the proof.
\end{proof}

\begin{definition}
Suppose $(\mathscr{A}, \ast, \alpha)$  is a Hom-rhizaform algebra over $\mathbb{F}.$ For a fixed\\ $z\in\mathscr{A}, ad_z(x)=z\ast x-x\ast z\, \forall x\in \mathscr{A}.$
Then, $ad_z$ is a $\alpha$-derivation referred to as  inner derivation of $\mathscr{A}$ where the $\ast$ is considered for the two products $\prec$ and $\succ.$ 
$Inner(\mathscr{A})$ is a usual notation for the set of all inner derivation of $\mathscr{A}$.
\end{definition}

\begin{proposition}
Suppose $(\mathscr{A},\ast, \alpha)$ is a Hom-Rhizaform algebra over a field $\mathbb{K}.$ For $z\in\mathscr{A},$ set\\ $ad_z(x)=z\prec x-x\succ z.$
Then, $ad_z$ is a derivation of $\mathscr{A}.$  
\end{proposition}
\begin{proof}
We need to show that $ad_z$ is an $\alpha$-derivation of $(\mathscr{A}, \ast, \alpha)$. For any $x,y\in \mathscr{A},$ we have
\begin{align*}
ad_z(x\prec y)&=\alpha(z)\succ(x\prec y)-(x\prec y)\prec \alpha(z)
\end{align*}
and
\begin{align*}
ad_z(x)\prec\alpha(y)+\alpha(x)\prec ad_z(y)&=(z\succ x-x\prec z)\prec\alpha(y)+\alpha(x)\prec(z\succ y-y\prec z)\\
&=\alpha(z)\succ(x\succ y)-\alpha(x)\prec(z\prec y)+\alpha(x)\prec(z\succ y)-\alpha(x)\prec(z\prec y)\\
&=\alpha(z)\succ(x\succ y)-\alpha(x)\prec(z\prec y)+\alpha(x)\prec(z\succ y)-\alpha(x)\prec(z\prec y)\\
&=\alpha(z)\succ(x\prec y)-(x\prec y)\prec\alpha(z).
\end{align*}
Hence $ad_z(x\prec y)=ad_z(x)\prec\alpha(y)+\alpha(x)\prec ad_z(y).$

On the other hand, 
\begin{align*}
ad_z(x\succ y)&=(x\succ y)\succ\alpha(z)-\alpha(z)\succ(x\succ y),
\end{align*}
and 
\begin{align*}
ad_z(x)\succ\alpha(y)+\alpha(x)\succ ad_z(y)&=(z\succ x-x\prec z)\succ\alpha(y)+\alpha(x)\succ(z\succ y-y\prec z)\\
&=\alpha(z)\succ(x\succ y)-(x\prec z)\succ\alpha(y)+(x\succ z)\succ\alpha(y)-(x\succ y)\succ\alpha(z)\\
&=\alpha(z)\succ(x\succ y)-(x\prec z)\succ\alpha(y)+(x\prec z)\succ\alpha(y)-(x\succ y)\prec\alpha(z)\\
&=\alpha(z)\succ(x\succ y)-(x\succ y)\succ\alpha(z).
\end{align*}
Thus, $ad_z(\succ y)=ad_z(x)\succ\alpha(y)+\alpha(x)\succ ad_z(y).$ Therefore, $ad_z$ is an inner-derivation of $\mathscr{A}.$
\end{proof}

 \begin{definition}
 Let $(\mathscr{A},\ast,\alpha)$ be a Hom-anti-associative algebra. A \textbf{bimodule} $(M, L, R, \beta)$ over $(\mathscr{A}, \ast, \alpha)$ consists of:
\begin{itemize}
  \item an $\mathbb{F}$-vector space $M$,
  \item a linear map $\beta : M \to M$,
  \item a linear left action $l : A \to \mathrm{End}_F(M)
$.
  \item a linear right action $r: A \to \mathrm{End}_F(M)$.
\end{itemize}
such that, for all $a, b \in \mathscr{A}$ and $m \in M$, the following identities are satisfied:
\begin{align}
l(\alpha(a))(l(b)(m)) &=-l(a \ast b)(\beta(m))\\
r(\alpha(b))(r(a)(m))&=-r(a \ast b) (\beta(m))\\
l(\alpha(a))(r(b)(m))&=-r(\alpha(b))(l(a)(m))\\
\beta(l(a)(m))&=l(\alpha(a))(\beta(m))\\
\beta(r(a)(m))&=r(\alpha(a))(\beta(m)).
\end{align} 
 In particular, $(\mathscr{A}, l_*, r_*, \alpha)$ is a bimodule over $(\mathscr{A}, *, \alpha)$, where $l_*, r_* : \mathscr{A} \times \mathscr{A} \to \mathscr{A}$ are defined by
\[l_*(x) (y) = r_*(y)( x) = x * y \quad \text{for all } x, y \in \mathscr{A}.
\]
\end{definition}  
\begin{lem}
 Let $(V, l, r,\beta)$ be a bimodule over an Hom-anti-associative algebra $\mathscr{A}$. Then $(V^*, r^*, l^*,\beta^*)$ is also a bimodule over an Hom-anti-associative algebra $\mathscr{A}$, where for all $x \in \mathscr{A}$, $u^* \in V^*$, $v \in V$, the actions $l^*, r^*\ and\ \beta^*$ are defined by:
\[\beta^*(u^*)(v)=u^*(\beta(v)),\quad
\langle l^*(x)u^*, v \rangle = \langle u^*, l(x)v \rangle, \quad 
\langle r^*(x)u^*, v \rangle = \langle u^*, r(x)v \rangle
\]
  $ \forall x\in \mathscr{A},\ u^*\in V^*,\ v \in V. $
\end{lem}
\begin{proof}
 \begin{itemize}
  \item[1.] For all $u^*$ and $v$,
  \[
  \langle \beta^* l^*(x) u^*,\, v \rangle 
  = \langle l^*(x) u^*,\, \beta(v) \rangle 
  = \langle u^*,\, l(x)\, \beta(v) \rangle.
  \]
  But $l(x)\beta = \beta\, l(\alpha(x))$, so:
  \[
  = \langle u^*,\, \beta\, l(\alpha(x))\, v \rangle 
  = \langle \beta^* u^*,\, l(\alpha(x))\, v \rangle 
  = \langle l^*(\alpha(x))\, \beta^* u^*,\, v \rangle.
  \]
  Therefore, $\beta^*\, l^*(x) = l^*(\alpha(x))\, \beta^*$.

  \item[2.] The same argument for $r^*$:
		
	\begin{align*}
\langle \beta^* r^*(x) u^*,\, v \rangle &=\langle r^*(x) u^*,\, \beta(v) \rangle\\
&=\langle u^*,\, r(x)\, \beta(v) \rangle=\langle u^*,\, \beta\, r(\alpha(x))\, v \rangle\\
&=\langle \beta^* u^*,\, r(\alpha(x))\, v \rangle=\langle r^*(\alpha(x))\, \beta^* u^*,\, v \rangle.
\end{align*}

  \item[3.] For $l^*(\mu(x,y))\, \beta^*$:
  \[
  \langle l^*(\mu(x,y))\, \beta^* u^*,\, v \rangle 
  = \langle \beta^* u^*,\, l(\mu(x,y))\, v \rangle 
  = u^*(\beta(l(\mu(x,y))\, v)).
  \]
  But $l(\mu(x,y))\, \beta = -\, l(\alpha(x))\, l(y)$, so:
  \[
  = u^*(-\, l(\alpha(x))\, l(y)\, v)
  = -\, u^*(l(\alpha(x))\, l(y)\, v)
  = -\, \langle l^*(\alpha(x))\, u^*,\, l(y)\, v \rangle
  = -\, \langle l^*(\alpha(x))\, l^*(y)\, u^*,\, v \rangle.
  \]

  Therefore,
  \[
  l^*(\mu(x,y))\, \beta^* = -\, l^*(\alpha(x))\, l^*(y).
  \]

  \item[4.] Similarly, for $r^*$:
  \[
  \langle r^*(\mu(x,y))\, \beta^* u^*,\, v \rangle 
  = \langle \beta^* u^*,\, r(\mu(x,y))\, v \rangle 
  = u^*(\beta(r(\mu(x,y))\, v)).
  \]
  But $r(\mu(x,y))\, \beta = -\, r(\alpha(y))\, r(x)$, so:
  \[
  = -\, u^*(r(\alpha(y))\, r(x)\, v)
  = -\, \langle r^*(\alpha(y))\, u^*,\, r(x)\, v \rangle
  = -\, \langle r^*(\alpha(y))\, r^*(x)\, u^*,\, v \rangle.
  \]

  Therefore,
  \[
  r^*(\mu(x,y))\, \beta^* = -\, r^*(\alpha(y))\, r^*(x).
  \]

  \item[5.] Finally,
  \[
  \langle r^*(\alpha(y))\, l^*(x) u^*,\, v \rangle 
  = \langle l^*(x) u^*,\, r(\alpha(y))\, v \rangle 
  = \langle u^*,\, l(x)\, r(\alpha(y))\, v \rangle.
  \]
  But $r(\alpha(y))\, l(x) = -\, l(\alpha(x))\, r(y)$, so:
  \[
  = \langle u^*,\, -\, l(\alpha(x))\, r(y)\, v \rangle 
  = -\, \langle l^*(\alpha(x))\, u^*,\, r(y)\, v \rangle 
  = -\, \langle l^*(\alpha(x))\, r^*(y)\, u^*,\, v \rangle.
  \]

  Therefore,
  \[
  r^*(\alpha(y))\, l^*(x) = -\, l^*(\alpha(x))\, r^*(y).
  \]
\end{itemize}
	This ends the proof.
\end{proof}
\begin{thm}
 Let \( \mathscr{A} \) be a vector space with two bilinear operations 
\[
\succ: \ \mathscr{A} \times \mathscr{A} \to \mathscr{A}, \quad \prec:\ \mathscr{A} \times \mathscr{A} \to \mathscr{A},\] and a linear map $\alpha:\ \mathscr{A}\to \mathscr{A}$. Define a bilinear operation \( \ast \) as
\begin{equation}
    x \ast y = x \succ y + x \prec y, \quad \forall x, y \in \mathscr{A}.
\end{equation} 
Then $(\mathscr{A}, \succ, \prec, \alpha)$ is a Hom-rhizaform algebra if and only if $(\mathscr{A}, *, \alpha)$ is an Hom-anti-associative algebra and $(\mathscr{A}, L_\ast R_\ast, \alpha)$ is 
a bimodule over $(\mathscr{A}, *, \alpha)$, where the two maps $L_\succ, R_\prec : \mathscr{A} \times \mathscr{A} \to \mathscr{A}$ are defined by
\[
L_\ast(x )( y) = x \succ y, \quad R_\ast(x )( y) = y \prec x, \quad \forall x, y \in \mathscr{A}.
\]
\begin{proof}
  From Eqs.~\eqref{eq0}, \eqref{req1}, \eqref{req2}, and \eqref{eq8}, it follows that \((\mathscr{A}, \succ, \prec, \alpha)\) is a Hom-rhizaform algebra if and only if \((\mathscr{A}, \succ, \prec, \alpha)\) is an Hom-anti-associative admissible algebra, that is, \((\mathscr{A}, * , \alpha)\) is an Hom-anti-associative algebra, and for all \(x, y, z \in \mathscr{A}\), the following equations hold:
\[
\alpha(x) \succ (y \succ z) = - (x * y) \succ \alpha(z),
\]
\[
(x \prec y) \prec \alpha(z) = - \alpha(x) \prec (y * z),
\]
\[
(x \succ y) \prec \alpha(z) = - \alpha(x) \succ (y \prec z).
\]
Hence, \((\mathscr{A}, L_{\ast}, R_{\ast},\alpha)\) is a bimodule.
\end{proof}
\end{thm}
\section{Hom-Rhizaform algebras and O-operators}
\begin{definition}
Let \((\mathscr{A}, \ast, \alpha)\) be an Hom-anti-associative algebra, and \((V,L,R, \beta)\) a bimodule over $(\mathscr{A}, \ast, \alpha)$ .
A linear map \(T: V \to \mathscr{A}\) is called an \emph{O-operator} of \((\mathscr{A}, \ast, \alpha)\) associated to \((V,L,R, \beta)\) if
\[
T \circ \beta = \alpha \circ T,\quad
and\quad 
T(u) * T(v) = T\left( L(T(u))(v) + R(T(v))(u) \right), \quad \forall u, v \in V.
\]
In particular, an O-operator $T$ of $(A, *, \alpha)$ associated to the bimodule $(A, L_*, R_*, \alpha)$ is called a Rota-Baxter operator of weight 0, 
that is, $T : A \to A$ is a linear map satisfying
\[
T(x) * T(y) = T\big(T(x) * y + x * T(y)\big), \quad \forall x, y \in A,
\]
and
\[
T \circ \alpha = \alpha \circ T.
\]
In these case, we call $(A, T, \alpha)$ a Rota-Baxter Hom-algebra of weight 0.

\end{definition}
\begin{thm}\label{th0}
Let $(\mathscr{A}, \cdot, \alpha)$ be a Hom-anti-associative algebra and $(V, L, R, \beta)$ be a bimodule. Suppose that 
$T : V \rightarrow \mathscr{A}$ is an $\mathcal{O}$-operator of $(\mathscr{A}, \cdot, \alpha)$ associated to $(V, L, R, \beta)$. Define two bilinear operations $\succ, \prec$ on $V$ respectively as
\begin{align}
u \succ v = L(T(u))( v), \quad u \prec v = R(T(v))(u), \quad \forall u, v \in V.\label{eq00}
\end{align}
Then $(V, \succ, \prec, \beta)$ is a Hom-rhizaform algebra. In this case, $T$ is a homomorphism of Hom-anti-associative algebras from the associated Hom-anti-associative algebra $(V, \ast, \beta)$ to $(\mathscr{A}, \cdot, \alpha)$. Furthermore, there is an induced Hom-rhizaform algebra structure on 
\[
T(V) = \{ T(u) \mid u \in V \} \subseteq \mathscr{A}
\]
given by
\begin{align}
    T(u) \succ T(v) = T(u \succ v), \quad T(u) \prec T(v) = T(u \prec v), \quad \forall u, v \in V,\label{eq01}
\end{align}
and $T$ is a homomorphism of Hom-rhizaform algebras.
\end{thm}
\begin{proof}
 \textbf{Left-hand side:}
\begin{align*}
(u \succ v + u \prec v) \succ \beta(w)
&= L( T(L(T(u))(v)))(\beta(w))+L( T(R(T(V))(U)))(\beta(w))\\
&= L\left( T(u) \cdot T(v) \right)(\beta(w))
\end{align*}
\textbf{Right-hand side:}
\begin{align*}
-\beta(u) \succ (v \succ w)
&= -L(T(\beta(u)))(L(T(v))(w))\\
&= -L(\alpha(T(u)))(L(T(v))(w))\\
\quad \text{(since } T \circ \beta = \alpha \circ T \text{)}
\end{align*}

Since \( (V, L, R, \beta) \) is a bimodule of the Hom-anti-associative algebra \( (\mathscr{A}, \cdot, \alpha) \), we have:
\[L\left( T(u) \cdot T(v) \right)(\beta(w))=-L(\alpha(T(u)))(L(T(v))(w)).
\]

Then
\[
(u \succ v + u \prec v) \succ \beta(w) = -\beta(u) \succ (v \succ w)
\]
Similarly, we have
\begin{align*}
(u \prec v) \prec \beta(w) &= -\beta(u) \prec (v \succ w + v \prec w), \\
(u \succ v) \prec \beta(w) &= -\beta(u) \succ (v \prec w).
\end{align*}
Thus, by Definition~\ref{eq0}, \((V, \succ, \prec, \beta)\) is a Hom-rhizaform algebra.
\[
\begin{aligned}
T(u \ast v) &= T(u \succ v + u \prec v) = T(u \succ v) + T(u \prec v) \\
&= T(L(T(u))(v))+ T(L(T(v))(u)) \\
&= T(u) \cdot T(v), \quad \forall u, v \in V,
\end{aligned}
\]
and the relation $
T \circ \beta = \alpha \circ T$
shows that \(T\) is a homomorphism of Hom anti-associative algebras from the associated Hom anti-associative algebra \((V, \ast, \beta)\) to \((\mathscr{A}, \cdot, \alpha)\).
Furthermore, for any $u,\ v,\ w \in V$, we can obtain
\begin{align*}
\alpha(T(u)) \succ (T(v) \succ T(w)) 
&= T(\beta(u)) \succ (T(v) \succ T(w)) \\
&= T(\beta(u) \succ (v \succ w)) \\
&= -T(u \succ v) \succ T(\beta(w)) - T(u \prec v) \succ T(\beta(w)) \\
&= -(T(u) \succ T(v)) \succ T(\beta(w)) - (T(u) \prec T(v)) \succ T(\beta(w)) \\
&= -\big( T(u) \succ T(v) + T(u) \prec T(v) \big) \succ T(\beta(w))\\ 
&= -\big( T(u) \succ T(v) + T(u) \prec T(v) \big) \succ \alpha(T(w)).
\end{align*}
Similarly, we can derive
\[
(T(u) \prec T(v)) \prec \alpha(T(w)) = -\alpha(T(u)) \prec (T(v) \succ T(w) + T(v) \prec T(w)),
\]
\[
(T(u) \succ T(v)) \prec \alpha(T(w)) = -\alpha(T(u)) \succ (T(v) \prec T(w)),
\]
which imply that \(T(V)\) is a Hom-rhizaform algebra. This completes the proof.
\end{proof}
\begin{cor}
 Let \((\mathscr{A}, *, \alpha)\) be a Hom-anti-associative algebra and \(R\) be a Rota-Baxter operator of weight 0. Then the triple \((\mathscr{A}, \succ, \prec, \alpha)\) 
is a Hom-rhizaform algebra, where
\[
x \succ y = R(x) * y, \quad x \prec y = x * R(y), \quad \forall x,y \in \mathscr{A}.
\]  
\end{cor}
\begin{thm}\label{th000}
 Let \( (\mathscr{A}, \ast, \alpha) \) be a Hom-anti-associative algebra. Then there exists a compatible Hom-rhizaform algebra structure  
on \( (\mathscr{A}, \ast, \alpha) \) if and only if there exists an invertible \( \mathcal{O} \)-operator of \( (\mathscr{A}, \ast, \alpha) \).
   \end{thm}
   \begin{proof}
   Suppose that $(\mathscr{A},\succ,\prec,\alpha)$ is a compatible Hom-rhizaform algebra structure on $(\mathscr{A},\ast,\alpha)$. Then 
\[
x \ast y = x \succ y + x \prec y = L_{\ast}(x)y + R_{\ast}(y)x, \quad \forall x,y \in \mathscr{A}.
\]
Hence the identity map $\mathrm{Id} : \mathscr{A} \to \mathscr{A}$ is an invertible $\mathcal{O}$-operator of $(\mathscr{A},\ast,\alpha)$ associated to the bimodule $(\mathscr{A}, L_{\ast}, R_{\ast}, \alpha)$. 
Conversely, suppose that $T : V \to \mathscr{A}$ is an invertible $\mathcal{O}$-operator of $(\mathscr{A},\ast,\alpha)$ associated to $(V, l, r, \beta)$. Then by theorem \ref{th0}, there exists a compatible rhizaform algebra structure on $V$ and $T(V)$ defined by equation \ref{eq00} and equation \ref{eq01} respectively. 
Let $x, y \in \mathscr{A}$. Then there exist $u, v \in V$ such that $x = T(u)$, $y = T(v)$. Hence we have
\begin{align*}
x \ast y = T(u) \ast T(v)
& = T\big( l(T(u))v + r(T(v))u \big)\\
&= T(u \succ v+ u \prec v)\\
&=T(u) \succ T(v) + T(u) \prec T(v) \\
&=x \succ y + x \prec y.
\end{align*}
So $(\mathscr{A},\succ,\prec,\alpha)$ is a compatible Hom-rhizaform algebra structure on $(\mathscr{A},\ast,\alpha)$.
 \end{proof}
\section{Hom-Rhizaform algebras and Connes cocycles}
 In this section, we show that Hom-rhizaform algebras can be obtained from nondegenerate Connes cocycles of Hom-anti
associative algebras.
\begin{definition}
 Let \( (\mathscr{A}, \ast, \alpha) \) be a \emph{Hom-anti-associative algebra}. 
A bilinear form \( B : \mathscr{A} \times \mathscr{A} \to \mathbb{K} \) is called a \emph{Connes cocycle} if it satisfies the condition:

\begin{align*}
B(a \ast b, \alpha(c)) + B(b \ast c, \alpha(a)) + B(c \ast a, \alpha(b)) = 0\\ 
\alpha \circ B = B \circ \alpha^{\otimes 2},\quad \forall \quad a, b, c \in \mathscr{A}.
\end{align*}

\end{definition}
\begin{thm}
Let \( (\mathscr{A}, \ast,\alpha) \) be an Hom-anti-associative algebra and let \( B \) be a nondegenerate Connes cocycle on \( (\mathscr{A}, \ast,\alpha) \). Then there exists a compatible Hom-rhizaform algebra structure \( (\mathscr{A}, \succ, \prec,\alpha) \) on \( (\mathscr{A}, \ast,\alpha) \) defined by:
\begin{align}
B(x \succ y, z) &= B(y, z \ast x), \\ 
B(x \prec y, z) &= B(x, y \ast z), \label{eq000}
\end{align}
$\forall x,\ y,\ z \in \mathscr{A}.$
\end{thm}
\begin{proof}
Define a linear map $T:\mathscr{A}\to \mathscr{A}^*$ by
$\langle T(x), y \rangle=B(x, y), \quad
\forall x,\ y,\ z \in \mathscr{A}.$\\
 Let $f, g \in \mathscr{A}^*$, and let $x = T^{-1}(f)$, $y = T^{-1}(g)$. Then
$$f = T(x), \quad g = T(y) \Rightarrow f(z) = B(x, z), \quad g(z) = B(y, z), \quad \forall z \in \mathscr{A}.$$

We want to show that:
\[
x \ast y = T^{-1}(R^*(x)(g) + L^*(y)(f)).
\]

We evaluate both sides on an element $z \in \mathscr{A}$. To do this, we use the  identity:
\[
B(x \ast y, z) = B(x, y \ast z) + B(y, z \ast x).
\]

Now observe:
\begin{align*}
R^*(x)(g)(z) &= g(z \ast x) = B(y, z \ast x), \\
L^*(y)(f)(z) &= f(y \ast z) = B(x, y \ast z),
\end{align*}
so that:
\[
R^*(x)(g)(z) + L^*(y)(f)(z) = B(y, z \ast x) + B(x, y \ast z) = B(x \ast y, z).
\]

Therefore, for all $z \in A$,
\[
B(x \ast y, z) = (R^*(x)(g) + L^*(y)(f))(z),
\]
which implies (by the non-degeneracy of $B$) that:
\[
x \ast y = T^{-1}(R^*(x)(g) + L^*(y)(f)).
\]
 Let $f \in A^*$, and let $x = T^{-1}(f)$, so that $f(y) = B(x, y)$.

Then:
\[
\alpha^*(f)(y) = f(\alpha(y)) = B(x, \alpha(y)) = B(\alpha(x), y) \quad \text{(by invariance)}.
\]

Therefore:
\[
\alpha^*(f) = T(\alpha(x)) = T(\alpha(T^{-1}(f))) \Rightarrow T^{-1}(\alpha^*(f)) = \alpha(T^{-1}(f)).
\]

Thus we obtain:
\[
T^{-1} \circ \alpha^* = \alpha \circ T^{-1}.
\]
we conclude that $T^{-1}$ is an $\mathcal{O}$-operator of $(\mathscr{A}, \ast,\alpha)$ associated to the bimodule $(\mathscr{A}^*, R^*, L^*,\alpha^*)$.\\
By Theorem~\ref{th000}, there is a compatible Hom-rhizaform algebra structure $\succ$, $\prec$ on $(\mathscr{A}, \ast, \alpha)$ given by:
\begin{align*}
x \succ y &= T^{-1}(R^*(x)(T(y))), \\
x \prec y &= T^{-1}(L^*(y)(T(x))), \quad \forall x, y \in \mathscr{A},
\end{align*}
which gives exactly Equation~\ref{eq000}.
\end{proof}
\section{Nilpotent Hom-rhizaform algebras}
\begin{definition}
 A \textit{Hom-Rhizaform algebra} $(\mathscr{A},\succ,\prec,\alpha)$ is said to be \textit{right nilpotent} (respectively, \textit{left nilpotent}) if it is \textit{right nilpotent} (respectively, \textit{left nilpotent}) as an algebra :\\ Let $A$ be an $n$-dimensional rhizaform algebra and $(e_1,e_2,\dots,e_n)$ a basis of $A$. An algebra structure on $A$ is determined by a set of structure constants $\alpha^{\,k}_{ij}$ and $\beta^{\,q}_{lp}$ defined by
\[
e_i \succ e_j = \sum_s \gamma^{\,s}_{ij}\,e_s, \qquad 
e_l \prec e_q = \sum_{t} \delta^{\,t}_{lp}\,e_t,
\qquad 1\le i,j,l,p\le n.
\]\\
 Let $\mathscr{M}$ and $\mathscr{N}$ be subsets of $\mathscr{A}$. We define
\[
\mathscr{M} \diamond \mathscr{N} := \mathscr{M} \succ \mathscr{N} + \mathscr{M} \prec \mathscr{N},\]
where\\
\begin{align*}
&\mathscr{M} \succ \mathscr{N} = \left\{ \sum_{i,j=1} \alpha_{ij} \, d_i \succ d_j \;\middle|\; \alpha_{ij} \in \mathbb{C},\, d_i \in \mathscr{M},\, d_j \in \mathscr{N}\right\},\\
&\mathscr{M} \prec \mathscr{N} = \left\{ \sum_{k,s=1} \beta_{ks} \, d_k \prec d_s \;\middle|\; \beta_{ks} \in \mathbb{C},\, d_k \in \mathscr{M},\, d_s \in \mathscr{N} \right\}.
\end{align*}
 consider the following series:\\
\begin{align*}
&\mathscr{A}^{\langle 1 \rangle} = \mathscr{A}, \quad \mathscr{A}^{\langle k+1 \rangle} = \mathscr{A}^{\langle k \rangle} \diamond \mathscr{A},\\
&\mathscr{A}^{\{1\}} = \mathscr{A}, \quad \mathscr{A}^{\{k+1\}} = \mathscr{A}\diamond \mathscr{A}^{\{k+1\}},\\
&\mathscr{A}^1 = \mathscr{A}, \quad \mathscr{A}^{k+1} = \mathscr{A}^1 \diamond \mathscr{A}^k + \mathscr{A}^2 \diamond \mathscr{A}^{k-1} + \cdots + \mathscr{A}^k \diamond \mathscr{A}^1.\\
\end{align*}
A \textit{Rhizaform algebra} $\mathscr{A}$ is said to be \textit{right nilpotent} (respectively, \textit{left nilpotent}), if there exists $k \in \mathbb{N}$ (respectively, $p \in \mathbb{N}$) such that $\mathscr{A}^{\langle k \rangle} = 0$ (respectively, $\mathscr{A}^{\{p\}} = 0$). 
A \textit{Rhizaform algebra} $\mathscr{A}$ is said to be \textit{nilpotent} if there exists $s \in \mathbb{N}$ such that $\mathscr{A}^{s} = 0$.\\
Note that for a multiplicative Hom-Lie algebra $\mathscr{A}$ we have $\alpha(\mathscr{A}^n)\subset\mathscr{A}^n$ for any n.
\end{definition}
\begin{lem}
 For any $g, h \in \mathbb{N}$ the following inclusions hold:
\[
\mathscr{A}^{\langle g \rangle} \diamond \mathscr{A}^{\langle h \rangle} \subseteq \mathscr{A}^{\langle g + h \rangle}, \quad 
\mathscr{A}^{g} \diamond \mathscr{A}^{h} \subseteq \mathscr{A}^{g + h}.
\]
\end{lem}
\begin{lem}
Let $(\mathscr{A}, \succ, \prec, \alpha)$ be a Hom-Rhizaform algebra. Then for any $g \in \mathbb{N}$ we have:
\[
\mathscr{A}^{\langle g \rangle} = \mathscr{A}^{\{g\}} = \mathscr{A}^{g}.
\]\\
In conclusion, the concepts of left nilpotency, right nilpotency, and over all nilpotency are equivalent in Hom-rhizaform algebras.
\end{lem}
\begin{thm}
 A Hom-Rhizaform algebra $(\mathscr{A}, \succ, \prec, \alpha)$ is nilpotent if and only if $(\mathscr{A}, \succ, \alpha)$ and $(\mathscr{A}, \prec, \alpha)$ are nilpotent.   
\end{thm}
\begin{proof}
  It is clear that from the nilpotency of $(\mathscr{A}, \succ, \prec, \alpha)$ follows the nilpotency of $(\mathscr{A}, \succ, \alpha)$ and $(\mathscr{A},\prec, \alpha)$.\\
  Now let $(\mathscr{A}, \succ, \alpha)$ and $(\mathscr{A},\prec, \alpha)$  be nilpotent algebras, i.e., there exist s,t$\in \mathscr{N}$ such that $\mathscr{A}^s_\succ=0$ and $\mathscr{A}^t_\prec=0$  Then, for $p \succ s+t-2$ we have $\mathscr{A}^p= 0$. 
\end{proof}
\begin{definition}
 A Hom-rhizaform algebra $(\mathscr{A}, \succ, \prec, \alpha)$ is 2-nilpotent if
\[
(x \mathbin{\ast_{i_1}} y) \mathbin{\ast_{i_2}} \alpha(z) = \alpha(x) \mathbin{\ast_{i_3}} (y \mathbin{\ast_{i_4}} z) = 0, \quad ,\ \forall\, \ast_{i_1}, \ast_{i_2}, \ast_{i_3}, \ast_{i_4} \in \{\succ, \prec\},\ x, y, z \in \mathscr{A}.
\]
\end{definition}
\begin{proposition}

 Let $(\mathscr{A}, \succ, \prec, \alpha)$ be a Hom-rhizaform algebra satisfying
\[
x \succ y + x \prec y = 0, \quad \text{for all } x, y \in \mathscr{A}.
\]
Then $(\mathscr{A}, \succ, \prec, \alpha)$ is a 2-nilpotent algebra. 
\end{proposition}
  \begin{proof}
  Let \( x, y, z \in \mathscr{A} \). Then, we obtain:
    \begin{align*}
\alpha(x) \succ (y \succ z) &= - (x \succ y + x \prec y) \succ \alpha(z) = 0, \\
(x \prec y) \prec z &= - \alpha(x) \prec (y \succ z + y \prec z) = 0, \\
\alpha(x) \prec (y \succ z) &= \alpha(x) \prec (y \succ z) + \alpha(x) \succ (y \succ z) = 0, \\
(x \prec y) \succ \alpha(z) &= (x \prec y) \succ \alpha(z) + (x \prec y) \prec \alpha(z) = 0, \\
(x \succ y) \succ \alpha(z) &= (x \succ y) \succ \alpha(z) + (x \prec y) \succ \alpha(z) \nonumber \\
&= (x \succ y + x \prec y) \succ \alpha(z) = 0, \\
\alpha(x) \prec (y \prec z) &= \alpha(x) \prec (y \prec z) + \alpha(x) \prec (y \succ z) \nonumber \\
&= \alpha(x) \prec (y \prec z + y \succ z) = 0, \\
(x \succ y) \prec \alpha(z) &= (x \succ y) \prec \alpha(z) + (x \succ y) \succ \alpha(z) = 0, \\
\alpha(x) \succ (y \prec z) &= (x \succ y) \prec \alpha(z) = 0.
\end{align*}
	This ends the proof.
 \end{proof}
\section{Hom-Rhizaform family Algebras}
\begin{definition}
 Let $(\Omega,\cdot)$ be a semi  group. A Hom-Rhizaform family algebra $(\mathscr{A},\{\prec_\gamma,\ \succ_\gamma;\ \gamma \in \Omega \},\ \alpha)$ is a $k$-vector space $\mathscr{A}$ equipped with a family of binary operations $(\prec_\gamma, \succ_\gamma;\ \gamma \in \Omega)$ and a lineair map $\alpha:\ \mathscr{A}\ \to \mathscr{A}$ such that for all $x, y, z \in \mathscr{A}$ and $\lambda,\ \omega \in \Omega$, the following conditions hold:
\begin{align}
\alpha \circ \prec_\omega &= \prec_\omega \circ \alpha^{\otimes 2}
\quad \text{and} \quad
\alpha \circ \succ_\omega = \succ_\omega \circ \alpha^{\otimes 2}\\
 (x \prec_\lambda y) \prec_\omega \alpha(z) &= -\alpha(x) \prec_{\lambda \cdot \omega} (y \prec_\omega z + y \succ_\lambda z)\\
    (x \succ_\lambda y) \prec_\omega \alpha(z)&=- \alpha(x) \succ_\lambda (y \prec_\omega z)\\
    (x \prec_\omega y + x \succ_\lambda y) \succ_{\lambda \cdot \omega} \alpha(z) &= -\alpha(x) \succ_\lambda (y \succ_\omega z)
\end{align}
\end{definition}
\begin{rem}
  \textit{When the semi group is taken to be the trivial monoid with one single 
element, a Hom-rhizaform family algebra is precisely a Hom-rhizaform algebra. 
Furthermore, a Hom-rhizaform family algebra reduces to a rhizaform family 
algebra when \( \alpha= \mathrm{Id}_A \).}  \end{rem}
\begin{definition}
Let $\Omega$ be a semi group. 
A Hom-anti-associative family algebras $(A, \cdot_{\lambda},\ \alpha)_{\lambda, \in \Omega}$ is a vector space $\mathscr{A}$ equipped with families of bilinear map
\[
\ast_{\lambda,\omega} : \mathscr{A} \times \mathscr{A} \to \mathscr{A} 
\]
 and a lineair map $ \alpha:\ \mathscr{A}\to \mathscr{A}$
such that
\begin{equation}
(x\ast_{\lambda,\omega}y)\ast_{\lambda\cdot\omega,\gamma}\alpha(z)=-\alpha(x)\ast_{\lambda,\omega\cdot\gamma}(y \ast_{\omega,\gamma} z), 
\end{equation}
for all $x, y, z \in \mathscr{A}$ and $\lambda, \omega,\gamma\in \Omega$.
\end{definition}
\begin{rem}
When the semi group is taken to be the trivial monoid with one single element,  a \emph{Hom-anti-associative family algebra} is precisely a \emph{Hom-anti-associative algebra}. Frurthermore, a \emph{Hom-anti-associative family algebra} reduces to a family of classical associative algebras parameterized by \(\Omega\) When \(\alpha = \mathrm{id}_\mathscr{A}\).
\end{rem}
\begin{proposition}
Let $(\mathscr{A}, \{ \prec_\omega,\ \succ_\omega \;|\; \omega \in \Omega \},\ \alpha)$ be a Hom-rhizaform family algebra. Define a bilinear operation $*_{\lambda,\Omega}$ by
  \[
  x *_{\lambda,\Omega} y := x \prec_\omega y + x \succ_\lambda y.
  \]
  Then $(\mathscr{A}, *_{\lambda,\Omega}, \alpha)$ is a Hom-anti-associative family algebra, called the \emph{associated Hom-anti-associative family algebra}. Furthermore, $(\mathscr{A}, \{ \prec_\omega,\ \succ_\omega \;|\; \omega \in \Omega \}, \alpha)$ is called a \emph{compatible Hom-rhizaform algebra structure} on $(\mathscr{A}, *_{\lambda,\Omega}, \alpha)$.
\end{proposition}
\begin{proof}
  \begin{align*}
(x\ast_{\lambda,\omega}y)\ast_{\lambda\cdot\omega,\gamma}\alpha(z)
&=(x \prec_\omega y + x \succ_\lambda y) \prec_\gamma \alpha(z) + (x \prec_\omega y + x \succ_\lambda y) \succ_{\lambda\omega} \alpha(z) \\
&=  (x \prec_\omega y) \prec_\gamma \alpha(z) + (x \succ_\lambda y) \prec_\gamma \alpha(z) - \alpha(x) \succ_\lambda (y \succ_\omega z) \\
&= -\alpha(x) \prec_{\omega\gamma} (y \prec_\gamma z + y \succ_\omega z) - \alpha(x) \succ_\lambda (y \prec_\gamma z) - \alpha(x) \succ_\lambda (y \succ_\omega z) \\
&= -\alpha(x) \prec_{\omega\gamma} (y \prec_\gamma z + y \succ_\omega z) - \alpha(x) \succ_\lambda (y \prec_\gamma z + y \succ_\omega z)\\
&=-\alpha(x)\ast_{\lambda,\omega\cdot\gamma}(y \ast_{\omega,\gamma} z).
\end{align*}  
This ends the proof.
\end{proof}
\begin{definition}
 A Rota-Baxter family of weight $0$ on a Hom-anti-associative algebra $(\mathscr{A}, *, \alpha)$ is a collection of linear operators $\{R_\lambda \mid \lambda \in \Omega\}$ on 
$(\mathscr{A}, *, \alpha)$ such that
\begin{align}
\alpha \circ R_\lambda &= R_\lambda \circ \alpha, \label{eq26}\quad \forall \lambda \in \Omega, \\
R_\lambda(a) * R_\omega(b) &= R_{\lambda\omega}\big(R_\lambda(a) * b + a * R_\omega(b)\big), \quad \forall a, b \in \mathscr{A},\ \lambda , \omega \in \Omega.
\end{align}
Then the quadruple $(\mathscr{A}, *, \alpha, \{R_\lambda \mid \lambda \in \Omega\})$ is called a \emph{Rota-Baxter family Hom-anti-associative algebra of weight $0$}. \\
If $\alpha=Id_\mathscr{A}$ we get the notion of a Rota-Baxter family of weight 0 on the associative algebra $(\mathscr{A}, *).$ 
\end{definition}
\begin{definition}
Let $(\mathscr{A}, \cdot, \alpha, \{R_\lambda, \lambda \in \Omega\})$ and $(\mathscr{B}, \top, \beta, \{P_\eta, \eta \in \Omega\})$ be two Rota-Baxter family Hom-algebras of weight $0$. 
A map $f : \mathscr{A} \to \mathscr{B}$ is called a \textit{Rota-Baxter family Hom-algebra morphism} if $f$ is Hom-algebras morphism and $f \circ R_\lambda = R_\lambda \circ f$ for each
$\lambda \in \Omega$.
\end{definition}
Now, we give a link between Rota-Baxter family Hom-algebras and ordinary Rota-Baxter Hom-algebra as follows.
Let $(\mathscr{A}, \cdot, \alpha)$ be a Hom-algebra. Consider the space $ \mathscr{A}\otimes \mathbb{K} \Omega$ with the linear map
\[
\tilde{\alpha} : \mathscr{A}\otimes \mathbb{K} \Omega \to \mathscr{A} \otimes \mathbb{K} \Omega, \quad 
\tilde{\alpha}(a \otimes \lambda) = \alpha(a) \otimes \lambda.
\]
Moreover, the multiplication on $\mathscr{A}$ induces a multiplication $\bullet$ on $\mathscr{A} \otimes \mathbb{K} \Omega$ by
\[
(a \otimes \lambda) \bullet (b \otimes \eta) := a \cdot b \otimes \lambda \eta.
\]
It is easy to see that $(\mathscr{A} \otimes \mathbb{K} \Omega, \bullet, \tilde{\alpha})$ is a Hom-algebra. 
Moreover, if $(\mathscr{A}, \cdot, \alpha, \{R_\lambda, \lambda \in \Omega\})$ is a Rota-Baxter family Hom-algebra of weight $0$ then, 
$(\mathscr{A} \otimes \mathbb{K} \Omega, \bullet, \tilde{\alpha}, R)$ is a Rota-Baxter Hom-algebra of weight $0$, 
where $R : \mathscr{A} \otimes \mathbb{K} \Omega \to \mathscr{A}\otimes \mathbb{K} \Omega$, 
$x \otimes \lambda \mapsto R_\lambda(x) \otimes \lambda$.
\begin{proposition}
A Rota-Baxter family $\{R_\lambda \mid \lambda \in \Omega\}$ of weight $0$ on a Hom-anti-associative algebra $(\mathscr{A}, *, \alpha)$ induces a Hom-rhizaform family algebra 
\[
(\mathscr{A}, *, \alpha, \{\prec_\lambda, \succ_\lambda \mid \lambda \in \Omega\})
\]
where the operations $\prec_\lambda$ and $\succ_\lambda$ are defined by
\begin{align}
a \prec_\lambda b &= a * R_\lambda(b), \\
a \succ_\lambda b &= R_\lambda(a) * b,
\end{align}
for all $\lambda \in \Omega$ and $a, b \in \mathscr{A}$.
\end{proposition}
\begin{proof}
  First, the multiplicativity of $\alpha$ with respect to $\ast$ and Condition \ref{eq26} imply
 the multiplicativity of$\alpha$ with respect to $\prec_\lambda$ and $\succ_\lambda$ for any $\lambda \in \Omega$. Next, for all $x,\ y, \ z \in \mathscr{A}$ and $\lambda,\omega \in \Omega$,  we compute  
 \begin{align*}
(x \prec_{\lambda} y) \prec_{\omega} \alpha(z) &= (x \cdot R_{\lambda}(y) ) \prec_{\omega} \alpha(z) \\
&= (x \cdot R_{\lambda}(y) ) \cdot R_{\omega}(\alpha(z)) \\
&= -\alpha(x) \cdot (R_{\lambda}(y) \cdot R_{\omega}(z))\\
&= -\alpha(x) \cdot R_{\lambda\omega}(y \cdot R_{\omega}(z) + R_{\lambda}(y) \cdot z)  \\
&=- \alpha(x) \prec_{\lambda\omega} (y \prec_{\omega} z + y \succ_{\lambda} z).
\end{align*}
\text{Similarly, we compute}
\begin{align*}
(x \succ_{\lambda} y) \prec_{\omega} \alpha(z) &= (R_{\lambda}(x) \cdot y) \prec_{\omega} \alpha(z) \\
&= (R_{\lambda}(x) \cdot y) \cdot R_{\omega}(\alpha(z))  \\
&= -R_{\lambda}(\alpha(x)) \cdot (y \cdot R_{\omega}(z)  \\
&= -R_{\lambda}(\alpha(x)) \cdot (y \prec_{\omega} z) = \alpha(x) \succ_{\lambda} (y \prec_{\omega} z).
\end{align*}
\text{ Finally we get}
\begin{align*}
\alpha(x) \succ_{\lambda} (y \succ_{\omega} z) &= \alpha(x) \succ_{\lambda} (R_{\omega}(y) \cdot z) = R_{\lambda}(\alpha(x)) \cdot (R_{\omega}(y) \cdot z) \\
&= -(R_{\lambda}(x) \cdot R_{\omega}(y)) \cdot \alpha(z) =- R_{\lambda\omega}(R_{\lambda}(x) \cdot y + x \cdot R_{\omega}(y))\cdot\alpha(z) \\
&= -(R_{\lambda}(x) \cdot y + x \cdot R_{\omega}(y))\succ_{\lambda\omega} \alpha(z) \\
&= -(x \succ_{\lambda} y + x \prec_{\omega} y) \succ_{\lambda\omega} \alpha(z).
\end{align*}
This ends the proof.
\end{proof}
\section{Classification}
This section describes the classification of  Hom-Rhizaform algebras  of dimension $\leq 3$ over the field $\mathbb{F}$ of characteristic $0$.
The sets of all multiplicative  and no multiplicative Hom-Rhizaform algebras of are denoted by $\mathscr{A}_p^q(m)$ and
$\mathscr{A}_p^q(nm)$ respectively.

Let $(\mathscr{A},\succ,\prec,\alpha,\beta)$ be an n-dimensional Hom-Rhizaform algebra, ${ei}$ be a basis of $\mathscr{A}$. For any $i, j \in\mathbb{N}, 1 \leq i, j \leq n$, let us put
$$e_i\succ e_j=\sum_{k=1}^n\gamma_{ij}^ke_k\quad ;\quad e_i\prec e_j=\sum_{k=1}^n\delta_{ij}^ke_k\quad ;\quad  \alpha(e_i)=\sum_{k=1}^n\alpha_{ji}e_j.$$
The axioms in Definition \ref{def0} are respectively equivalent to
$$\left\{\begin{array}{c}
\begin{array}{ll}
\gamma_{ij}^k\alpha_{qk}-\alpha_{ki}\alpha_{jp}\gamma_{kp}^q=0,\\
\delta_{ij}^k\alpha_{qk}-\alpha_{ki}\alpha_{jp}\delta_{kp}^q=0,
\end{array}
\end{array}\right.$$

$$\left\{\begin{array}{c}
\begin{array}{ll}
\delta_{ij}^p\alpha_{qk}\delta_{pq}^r+\gamma_{ij}^p\alpha_{qk}\gamma_{pq}^r+\alpha_{pi}\gamma_{jk}^q\delta_{pq}^r=0,\\
\alpha_{pi}\delta_{jk}^q\gamma_{pq}^r+\alpha_{pi}\gamma_{jk}^p\gamma_{pq}^r+\gamma_{ij}^p\alpha_{qk}\gamma_{pq}^r=0,\\
\alpha_{pi}\gamma_{jk}^q\delta_{pq}^r+\delta_{ij}^p\alpha_{qk}\gamma_{pq}^r=0.
\end{array}
\end{array}\right.$$

\begin{thm}
The isomorphism class of 2-dimensional Hom-Rhizaform algebras is given by the following representatives:
\begin{itemize}
\item
$
\mathscr{A}_1(m):\left\{
    \begin{array}{ll}
      e_2\succ e_2=e_1&\\
     e_2\prec e_2=e_1&\\
    \end{array}
\right.\quad
\begin{array}{ll}
      \alpha(e_1)=e_1&\\
     \alpha(e_2)=e_1+e_2&
    \end{array}$
\item
$\mathscr{A}_2(m):\left\{
    \begin{array}{ll}
      e_1\succ e_2=e_1& e_1\prec e_2=e_1\\
     e_2\succ e_1=e_1&\\
      e_2\succ e_2=e_1&
    \end{array}
\right.\quad
\begin{array}{ll}
      \alpha(e_2)=e_1.&
    \end{array}$

\item
$\mathscr{A}_3(m):\left\{
    \begin{array}{ll}
      e_1\succ e_2=e_1& e_1\prec e_2=e_1\\
     e_2\succ e_1=e_1&e_2\prec e_1=e_1\\
      e_2\succ e_2=e_1&
    \end{array}
\right.\quad
\begin{array}{ll}
      \alpha(e_2)=e_1.&
    \end{array}$
    
 \item
$\mathscr{A}_4(nm):\left\{
    \begin{array}{ll}
      e_1\succ e_2=e_1& e_2\prec e_1=e_1\\
     e_2\succ e_1=e_1&e_2\prec e_1=e_1\\
      e_2\succ e_2=e_1&e_2\prec e_2=e_1
    \end{array}
\right.\quad
\begin{array}{ll}
      \alpha(e_2)=e_1.&
    \end{array}$   
\item
$\mathscr{A}_5(nm):\left\{
    \begin{array}{ll}
      e_1\succ e_1=e_2& e_1\prec e_1=e_2\\
			e_1\succ e_2=e_2&e_1\prec e_2=e_2\\
			e_2\succ e_1=e_2&e_2\prec e_1=e_2
    \end{array}
\right.\quad
\begin{array}{ll}
    \alpha(e_2)=e_2.&
    \end{array}$   

\item
$\mathscr{A}_6(nm):\left\{
    \begin{array}{ll}
     e_1\succ e_2=e_1& e_1\prec e_2=e_1\\
			e_2\succ e_1=e_1&e_2\prec e_1=e_1\\
			e_2\succ e_2=e_1&e_2\prec e_2=e_1
    \end{array}
\right.\quad
\begin{array}{ll}
      \alpha(e_1)=e_1&
    \end{array}$ 
						
\item
$\mathscr{A}_7(m):\left\{
    \begin{array}{ll}
      e_1\succ e_1=e_2&\\
     e_1\prec e_1=e_2&
    \end{array}
\right.\quad
\begin{array}{ll}
      \alpha(e_1)=e_1&\\
    \alpha(e_2)=e_2.&
    \end{array}$ 		      

\end{itemize}
\end{thm}

\begin{thm}
The isomorphism class of 3-dimensional Hom-Rhizaform algebras is given by the following representatives:
\begin{itemize}
\item
$
\mathscr{A}_1(m):\left\{
    \begin{array}{ll}
      e_1\succ e_2=e_1+e_3&  e_2\prec e_1=e_1+\frac{3}{2}e_3\\
    e_2\succ e_2=e_1+e_3&  e_2\prec e_2=e_1+e_3\\
    e_2\succ e_3=e_1+e_3&  e_2\prec e_3=e_1+e_3\\
    e_3\succ e_2=e_1+e_3&  e_3\prec e_2=e_1+e_3\\
    \end{array}
\right.\quad
\begin{array}{ll}
      \alpha(e_2)=e_1&\\
     \alpha(e_3)=e_2.&
    \end{array}$
\item    
$
\mathscr{A}_2(m):\left\{
    \begin{array}{ll}
      e_1\succ e_2=e_1&  e_2\prec e_1=e_1\\
    e_2\succ e_1=e_1&  e_2\prec e_2=e_1\\
    e_2\succ e_3=e_1&  e_2\prec e_3=e_1\\
    e_3\succ e_2=e_1&  e_3\prec e_2=e_1\\
    \end{array}
\right.\quad
\begin{array}{ll}
      \alpha(e_2)=e_1&\\
     \alpha(e_3)=e_3.&
    \end{array}$
\item    
$
\mathscr{A}_3(nm):\left\{
    \begin{array}{ll}
      e_1\succ e_1=e_2&  e_1\prec e_3=e_2\\
    e_2\succ e_2=e_2&  e_2\prec e_2=e_2\\
    e_3\succ e_1=e_2&  e_3\prec e_2=e_2\\
    e_3\succ e_3=e_2&  e_3\prec e_3=e_2\\
    \end{array}
\right.\quad
\begin{array}{ll}
      \alpha(e_2)=e_1&\\
     \alpha(e_3)=e_3.&
    \end{array}$

 \item    
$
\mathscr{A}_4(nm):\left\{
    \begin{array}{ll}
      e_1\succ e_1=e_3&  e_1\prec e_1=e_3\\
    e_1\succ e_2=e_3&  e_1\prec e_2=\eta e_3\\
    e_2\succ e_1=e_3&  e_2\prec e_2=e_2\\
    e_2\succ e_2=\frac{1}{4}e_3&  e_3\prec e_2=e_3\\
    \end{array}
\right.\quad
\begin{array}{ll}
      \alpha(e_2)=e_1&\\
     \alpha(e_3)=e_3.&
    \end{array}$ 

  \item    
$
\mathscr{A}_5(nm):\left\{
    \begin{array}{ll}
      e_1\succ e_2=e_3&  e_1\prec e_1=e_3\\
    e_2\succ e_1=e_3&  e_1\prec e_2=e_3\\
    e_2\succ e_2=e_3&  e_2\prec e_2=e_2\\
    e_3\succ e_3=e_3&  e_3\prec e_2=e_3\\
    \end{array}
\right.\quad
\begin{array}{ll}
      \alpha(e_1)=e_1&\\
     \alpha(e_2)=e_1+e_2.&
    \end{array}$  
  \item    
$
\mathscr{A}_6(nm):\left\{
    \begin{array}{ll}
      e_1\succ e_3=e_1+e_2&  e_1\prec e_3=e_1+e_2\\
    e_2\succ e_3=e_1+e_2&  e_2\prec e_3=e_1+e_2\\
    e_3\succ e_1=e_1-e_2&  e_3\prec e_1=e_1+e_2\\
    e_3\succ e_2=e_1+e_2&  e_3\prec e_2=e_1+e_2\\
    e_3\succ e_3=e_2&  e_3\prec e_3=\eta e_1+e_2\\
    \end{array}
\right.\quad
\begin{array}{ll}
      \alpha(e_1)=e_1&\\
     \alpha(e_2)=e_1+e_2.&
    \end{array}$ 
 \item    
$
\mathscr{A}_7 (nm):\left\{
    \begin{array}{ll}
      e_1\succ e_1=e_3&  e_1\prec e_1=e_3\\
    e_1\succ e_2=e_3&  e_1\prec e_2=e_3\\
    e_2\succ e_1=e_3&  e_2\prec e_1=e_3\\
    e_2\succ e_2=-e_3&  e_2\prec e_2=\eta e_3\\
    \end{array}
\right.\quad
\begin{array}{ll}
      \alpha(e_1)=e_1\\
     \alpha(e_2)=e_1+e_2\\
      \alpha(e_3)=e_3.
    \end{array}$    
\item    
$
\mathscr{A}_8 (nm):\left\{
    \begin{array}{ll}
      e_1\succ e_3=e_2&  e_1\prec e_3=e_2\\
    e_3\succ e_1=-e_2&  e_3\prec e_1=e_2\\
    e_3\succ e_3=e_2&  e_3\prec e_3=\frac{1}{4}e_2\\
    \end{array}
\right.\quad
\begin{array}{ll}
      \alpha(e_1)=e_1\\
     \alpha(e_2)=e_1+e_2\\
      \alpha(e_3)=e_2+e_3.
    \end{array}$    
    
\item    
$
\mathscr{A}_9 (nm):\left\{
    \begin{array}{ll}
      e_1\succ e_1=e_2&  e_1\prec e_3=-e_2\\
    e_1\succ e_3=e_2&  e_2\prec e_3=e_2\\
		 e_2\succ e_3=e_2&  e_3\prec e_1=e_2\\
		 e_3\succ e_2=-e_2&  e_3\prec e_2=e_2\\
    e_3\succ e_3=\eta e_2&  e_3\prec e_3=\frac{1}{2}e_2\\
    \end{array}
\right.\quad
\begin{array}{ll}
      \alpha(e_1)=e_1\\
     \alpha(e_2)=e_2.
    \end{array}$    
		
\item    
$
\mathscr{A}_{10} (nm):\left\{
    \begin{array}{ll}
      e_1\succ e_2=e_3&  e_1\prec e_2=e_2\\
    e_2\succ e_1=-e_3&  e_2\prec e_1=e_3\\
		 e_2\succ e_2=e_3&  e_2\prec e_2=-2e_3\\
		 e_3\succ e_3=e_3&  e_3\prec e_3=e_3\\
    \end{array}
\right.\quad
\begin{array}{ll}
      \alpha(e_1)=e_1\\
     \alpha(e_2)=e_2.
    \end{array}$    		
    		
\item    
$
\mathscr{A}_{11} (nm):\left\{
    \begin{array}{ll}
      e_1\succ e_2=e_1&  e_1\prec e_2=e_1\\
			 e_2\succ e_1=e_1&  e_2\prec e_1=e_1\\
			 e_2\succ e_2=e_1&  e_2\prec e_2=e_1\\
    e_2\succ e_3=e_1&  e_3\prec e_2=e_1\\
		 e_3\succ e_2=e_1&  e_3\prec e_3=e_1\\
    \end{array}
\right.\quad
\begin{array}{ll}
      \alpha(e_1)=e_1\\
     \alpha(e_3)=e_3.
    \end{array}$ 
					
\item    
$
\mathscr{A}_{12} (nm):\left\{
    \begin{array}{ll}
      e_1\succ e_1=e_2&  e_1\prec e_1=e_2\\
			 e_1\succ e_3=e_2&  e_1\prec e_3=e_2\\
			 e_2\succ e_2=e_2&  e_2\prec e_2=e_2\\
    e_3\succ e_1=e_2&  e_3\prec e_1=e_2\\
		 e_3\succ e_3=e_2&  e_3\prec e_3=-e_2\\
    \end{array}
\right.\quad
\begin{array}{ll}
      \alpha(e_1)=e_1\\
     \alpha(e_3)=e_3.
    \end{array}$ 					
    						
\item    
$
\mathscr{A}_{13} (nm):\left\{
    \begin{array}{ll}
      e_1\succ e_1=e_1+e_3&  e_1\prec e_1=e_1+e_3\\
			 e_2\succ e_2=e_1+e_3&  e_2\prec e_2=e_1-e_3\\
    \end{array}
\right.\quad
\begin{array}{ll}
      \alpha(e_2)=e_2\\
     \alpha(e_3)=e_3.
    \end{array}$ 
							
\item    
$
\mathscr{A}_{14} (nm):\left\{
    \begin{array}{ll}
      e_1\succ e_1=e_1+e_3&  e_1\prec e_1=e_1+e_3\\
			 e_2\succ e_2=e_2&  e_2\prec e_2=e_2\\
    \end{array}
\right.\quad
\begin{array}{ll}
      \alpha(e_2)=e_2\\
     \alpha(e_3)=e_3.
    \end{array}$ 		
							
\item    
$
\mathscr{A}_{15} (nm):\left\{
    \begin{array}{ll}
      e_1\succ e_1=e_2&  e_1\prec e_1=e_2\\
		e_3\succ e_1=-2e_2&  e_1\prec e_3=e_2\\
		e_3\succ e_3=e_2&  e_3\prec e_3=e_2\\
    \end{array}
\right.\quad
\begin{array}{ll}
      \alpha(e_1)=e_1\\
			\alpha(e_2)=-2e_2\\
     \alpha(e_3)=e_3.
    \end{array}$ 

\item    
$
\mathscr{A}_{16} (m):\left\{
    \begin{array}{ll}
      e_1\succ e_1=e_2&  e_1\prec e_1=e_2\\
			  e_1\succ e_3=e_2&  e_1\prec e_3=e_2\\
		e_3\succ e_1=e_2&  e_3\prec e_1=e_2\\
		e_3\succ e_3=e_2&  e_3\prec e_3=e_2\\
    \end{array}
\right.\quad
\begin{array}{ll}
      \alpha(e_1)=e_1\\
			\alpha(e_2)=e_2\\
     \alpha(e_3)=e_3.
    \end{array}$ 											
\end{itemize}
\end{thm}

\begin{thm}
The Connes cocycle of 2-dimensional Hom-Rhizaform algebras is given by the following representatives:

$
\mathscr{A}_{1}:
    \begin{array}{ll}
     \omega(e_2,e_1)=c_{21}^1e_1\\
		  \omega(e_2,e_2)=c_{22}^1e_1+c_{21}^1e_2
    \end{array}$ ;
$
\mathscr{A}_{2}:
    \begin{array}{ll}
     \omega(e_1,e_2)=c_{12}^1e_1\\
		    \omega(e_2,e_1)=c_{21}^1e_1\\
		  \omega(e_2,e_2)=c_{22}^1e_1
    \end{array}$ ;
$
\mathscr{A}_{3}:
    \begin{array}{ll}
     \omega(e_1,e_2)=c_{12}^1e_1&\\
		    \omega(e_2,e_1)=c_{21}^1e_1&\\
		  \omega(e_2,e_2)=c_{22}^1e_1&
    \end{array}$	;
$
\mathscr{A}_{4}:
    \begin{array}{ll}
     \omega(e_1,e_2)=c_{12}^1e_1&\\
		    \omega(e_2,e_1)=c_{21}^1e_1&\\
		  \omega(e_2,e_2)=c_{22}^1e_1&
    \end{array}$		;		
	
	$
\mathscr{A}_{5}:
    \begin{array}{ll}
     \omega(e_1,e_1)=c_{11}^1e_1&\\
		 \omega(e_1,e_2)=c_{12}^1e_1&\\
		  \omega(e_2,e_1)=c_{21}^1e_1&
    \end{array}$ ;
$
\mathscr{A}_{6}:
    \begin{array}{ll}
     \omega(e_1,e_2)=c_{12}^2e_2&\\
		    \omega(e_2,e_1)=c_{21}^2e_2&\\
		  \omega(e_2,e_2)=c_{22}^2e_2&
    \end{array}$ ;
$
\mathscr{A}_{7}:
    \begin{array}{ll}
     \omega(e_1,e_1)=c_{11}^1e_1+c_{11}^2e_2&\\
		    \omega(e_1,e_2)=c_{12}^1e_1+c_{12}^2e_2&
    \end{array}$.		
\end{thm}

\begin{thm}
The Connes cocycle of 3-dimensional Hom-Rhizaform algebras is given by the following representatives:

$
\mathscr{A}_{1}:
    \begin{array}{ll}
     \omega(e_1,e_3)=c_{13}^1e_1\\
		  \omega(e_2,e_3)=c_{23}^1e_1\\
		  \omega(e_3,e_3)=c_{33}^1e_1
    \end{array}$
$\mathscr{A}_{2}:
    \begin{array}{ll}
     \omega(e_1,e_2)=c_{12}^1e_1\\
		  \omega(e_2,e_2)=c_{22}^1e_1\\
		  \omega(e_3,e_2)=c_{32}^1e_1
    \end{array}$
$
\mathscr{A}_{3}:
    \begin{array}{ll}
     \omega(e_1,e_2)=c_{12}^1e_1\\
		    \omega(e_2,e_2)=c_{22}^1e_1\\
		  \omega(e_3,e_2)=c_{32}^1e_1+c_{31}^1e_2\\
			 \omega(e_3,e_3)=c_{33}^3e_3
    \end{array}$	
$
\mathscr{A}_{4}:
    \begin{array}{ll}
     \omega(e_1,e_2)=c_{12}^1e_1&\\
		    \omega(e_2,e_2)=c_{22}^1e_1&\\
		  \omega(e_3,e_2)=c_{32}^1e_1&
    \end{array}$				
	$
\mathscr{A}_{5}:
    \begin{array}{ll}
     \omega(e_1,e_3)=c_{13}^3e_3&\\
		 \omega(e_2,e_3)=c_{23}^3e_3&\\
		  \omega(e_3,e_3)=c_{33}^3e_3&
    \end{array}$
$
\mathscr{A}_{6}:
    \begin{array}{ll}
     \omega(e_1,e_3)=c_{13}^3e_3&\\
		  \omega(e_2,e_3)=c_{23}^3e_3&\\
		  \omega(e_3,e_1)=c_{31}^3e_3&\\
			\omega(e_3,e_2)=c_{32}^3e_3&\\
			\omega(e_3,e_3)=c_{33}^3e_3&
    \end{array}$
$
\mathscr{A}_{7}:
    \begin{array}{ll}
     \omega(e_i,e_j)=0&
    \end{array}$		
		$
\mathscr{A}_{8}:
    \begin{array}{ll}
     \omega(e_i,e_j)=0&
    \end{array}$
			
$
\mathscr{A}_{9}:
    \begin{array}{ll}
     \omega(e_1,e_3)=c_{13}^3e_3&\\
		  \omega(e_2,e_3)=c_{23}^3e_3&\\
		  \omega(e_3,e_1)=c_{31}^3e_3&\\
			\omega(e_3,e_2)=c_{32}^3e_3&\\
			\omega(e_3,e_3)=c_{33}^3e_3&
    \end{array}$
$
\mathscr{A}_{10}:
    \begin{array}{ll}
     \omega(e_1,e_3)=c_{13}^3e_3&\\
		  \omega(e_2,e_3)=c_{23}^3e_3&\\
			\omega(e_3,e_3)=c_{33}^3e_3&
    \end{array}$
$
\mathscr{A}_{11}:
    \begin{array}{ll}
     \omega(e_1,e_2)=c_{12}^2e_2&\\
		  \omega(e_2,e_2)=c_{22}^2e_2&\\
		  \omega(e_3,e_1)=c_{31}^1e_1+c_{31}^3e_3&\\
			\omega(e_3,e_2)=c_{32}^3e_2&\\
			\omega(e_3,e_3)=c_{33}^1e_1+c_{33}^3e_3&
    \end{array}$
									
$
\mathscr{A}_{12}:
    \begin{array}{ll}
     \omega(e_1,e_2)=c_{12}^2e_2&\\
		  \omega(e_2,e_2)=c_{22}^2e_2&\\
		  \omega(e_3,e_1)=c_{31}^1e_1+c_{31}^3e_3&\\
			\omega(e_3,e_2)=c_{32}^3e_2&\\
			\omega(e_3,e_3)=c_{33}^1e_1+c_{33}^3e_3&
    \end{array}$
$\mathscr{A}_{13}:
    \begin{array}{ll}
     \omega(e_1,e_1)=c_{11}^1e_1\\
		  \omega(e_2,e_1)=c_{21}^1e_1\\
			\omega(e_3,e_1)=c_{31}^1e_1
    \end{array}$
$\mathscr{A}_{14}:
    \begin{array}{ll}
     \omega(e_1,e_1)=c_{11}^1e_1\\
		  \omega(e_2,e_1)=c_{21}^1e_1\\
			\omega(e_3,e_1)=c_{31}^1e_1
    \end{array}$		

$\mathscr{A}_{15}:
    \begin{array}{ll}
     \omega(e_1,e_1)=c_{11}^1e_1+c_{11}^3e_3\\
		\omega(e_1,e_2)=c_{12}^2e_2\\
		\omega(e_1,e_3)=c_{13}^1e_1+c_{13}^3e_3\\
    \end{array}	
		 \begin{array}{ll}
     \omega(e_3,e_1)=c_{31}^1e_1+c_{31}^3e_3\\
		\omega(e_3,e_2)=c_{32}^2e_2\\
		\omega(e_3,e_3)=c_{33}^1e_1+c_{33}^3e_3.
    \end{array}$
							
$\mathscr{A}_{16}:
    \begin{array}{ll}
     \omega(e_1,e_1)=c_{11}^1e_1+c_{11}^2e_2+c_{11}^3e_3\\
		\omega(e_1,e_2)=c_{12}^1e_1+c_{12}^2e_2+c_{12}^3e_3\\
		\omega(e_1,e_3)=c_{13}^1e_1+c_{13}^2e_2+c_{13}^3e_3\\
    \end{array}	
		 \begin{array}{ll}
     \omega(e_3,e_1)=c_{31}^1e_1+c_{31}^2e_2+c_{31}^3e_3\\
		\omega(e_3,e_2)=c_{32}^1e_1+c_{32}^2e_2+c_{32}^3e_3\\
		\omega(e_3,e_3)=c_{33}^1e_1+c_{33}^2e_2+c_{33}^3e_3.
    \end{array}$							
\end{thm}

\end{document}